\documentclass[11pt,a4paper]{amsproc}
\usepackage{amssymb}
\usepackage[hyphens]{url} 
\usepackage[dvips]{graphicx}
\usepackage{graphicx,epstopdf} 
\usepackage[caption=false]{subfig} 
\theoremstyle{plain}
\usepackage{tikz}
\usetikzlibrary{arrows.meta,calc}
\usepackage{upgreek}
\usepackage{pagecolor}
\usepackage{booktabs}
\pagecolor{white}
\usepackage{nccmath}
\usepackage{mathtools}
\usepackage{enumitem}
 \newtheorem{theorem}{Theorem}[section]

\theoremstyle{definition}

\theoremstyle{remark}
\usepackage{xcolor}
 
 \numberwithin{equation}{section}
\renewcommand{\le}{\leqslant}
\renewcommand{\ge}{\geqslant}

\usepackage{geometry}
\usepackage{pifont}
\makeatletter
\renewcommand\subsubsection{%
  \@startsection{subsubsection}{3}{\z@}%
    {1.25ex plus .5ex minus .2ex}
    {-0.8em}
    {\normalfont\small\bfseries}
}
\makeatother

\makeatletter
\newcommand\Pimathsymbol[3][\mathord]{%
  #1{\@Pimathsymbol{#2}{#3}}}
\def\@Pimathsymbol#1#2{\mathchoice
  {\@Pim@thsymbol{#1}{#2}\tf@size}
  {\@Pim@thsymbol{#1}{#2}\tf@size}
  {\@Pim@thsymbol{#1}{#2}\sf@size}
  {\@Pim@thsymbol{#1}{#2}\ssf@size}}
\def\@Pim@thsymbol#1#2#3{%
  \mbox{\fontsize{#3}{#3}\Pisymbol{#1}{#2}}}
\makeatother

\DeclareFontFamily{U}{psyr}{}
\DeclareFontShape{U}{psyr}{m}{n}{<-> psyr }{}

\DeclareFontFamily{U}{psyro}{}
\DeclareFontShape{U}{psyro}{m}{n}{<-> psyro }{}
\newcommand{\pilambdait}{\Pimathsymbol[\mathord]{psyro}{"6C}}
\usepackage[boxruled,linesnumbered]{algorithm2e}

\usepackage{amsmath}
\usepackage{contour}
\DeclareMathAlphabet{\mathpzc}{OT1}{pzc}{m}{it}

\usepackage{xparse}

\NewDocumentCommand{\pd}{ O{} m m }{%
  \frac{\partial^{#1} #2}{\partial #3^{#1}}%
}

\renewcommand{\d}{\textrm{d}}

\newcommand{\dx}{\text{d} x}

\DeclareSymbolFont{uplargesymbols}{OMX}{mdbch}{m}{n}
\SetSymbolFont{uplargesymbols}{bold}{OMX}{mdbch}{b}{n}
\DeclareMathSymbol{\upintop}{\mathop}{uplargesymbols}{82}
\DeclareMathSymbol{\upointop}{\mathop}{uplargesymbols}{"48}
\DeclareFontEncoding{MDB}{}{}
\DeclareFontFamily{MDB}{mdbch}{}
\DeclareFontShape{MDB}{mdbch}{m}{n}{ <->s * [0.8]  mdbchrmb }{}
\DeclareFontShape{MDB}{mdbch}{b}{n}{ <->s * [0.8]  mdbchbmb }{}
\DeclareFontShape{MDB}{mdbch}{bx}{n}{<->ssub * mdbch/b/n}{}
\DeclareFontSubstitution{MDB}{cmr}{m}{n}
\DeclareSymbolFont{mathdesignB}{MDB}{mdbch}{m}{n}%
\SetSymbolFont{mathdesignB}{bold}{MDB}{mdbch}{b}{n}%
\DeclareMathSymbol{\upintclockwise}{\mathop}{mathdesignB}{128}
\DeclareMathSymbol{\upointclockwise}{\mathop}{mathdesignB}{130}
\DeclareMathSymbol{\upointctrclockwise}{\mathop}{mathdesignB}{132}
\DeclareMathSymbol{\upoiint}{\mathop}{mathdesignB}{134}
\DeclareMathSymbol{\upoiiint}{\mathop}{mathdesignB}{136}

\makeatletter
\newcommand{\upint}{\DOTSI\upintop\ilimits@}
\newcommand{\upoint}{\DOTSI\upointop\ilimits@}
\makeatother
\makeatletter
\def\upintkern@{\mkern-7mu\mathchoice{\mkern-3.5mu}{}{}{}}
\def\upintdots@{\mathchoice{\mkern-4mu\@cdots\mkern-4mu}%
 {{\cdotp}\mkern1.5mu{\cdotp}\mkern1.5mu{\cdotp}}%
 {{\cdotp}\mkern1mu{\cdotp}\mkern1mu{\cdotp}}%
 {{\cdotp}\mkern1mu{\cdotp}\mkern1mu{\cdotp}}}

\newcommand{\UpMultiIntegral}[1]{%
  \edef\upints@c{\noexpand\upintop
    \ifnum#1=\z@\noexpand\upintdots@\else\noexpand\upintkern@\fi
    \ifnum#1>\tw@\noexpand\upintop\noexpand\upintkern@\fi
    \ifnum#1>\thr@@\noexpand\upintop\noexpand\upintkern@\fi
    \noexpand\upintop
    \noexpand\ilimits@
  }%
  \futurelet\@let@token\upints@a
}
\makeatother
\usepackage[T1]{fontenc}
\usepackage{lmodern}

\title{\textbf{A Spectral Low-Mode Reduced Method for Elliptic Problems}}

\author{Prosper Torsu}

\date{\today}
\begin{document}
\maketitle
\begin{center}
\small
Department of Mathematics, California State University Bakersfield\\
\texttt{ptorsu@csub.edu}
\end{center}

\begin{abstract}
We develop a spectral low--mode reduced solver for second--order elliptic
boundary value problems with spatially varying diffusion coefficients. The
approach projects standard finite difference or finite element
discretizations onto a global coarse space spanned by the lowest Dirichlet
Laplacian eigenmodes, yielding an analytic reduced model that requires no
training data and preserves coefficient heterogeneity through an exact
Galerkin projection. The reduced solution is energy--optimal in the selected
subspace and, for $H^2$-regular solutions, the truncation error associated
with discarded modes satisfies a $\sqrt{\log M}/M$ decay in the
$H_0^1$ norm. For uniformly stable reduced bases, the projected operator is
well conditioned with respect to mesh refinement, and numerical experiments
corroborate the predicted accuracy and demonstrate meaningful speedups over
sparse direct solvers, with favorable performance relative to multigrid and
deflation-based Krylov methods for heterogeneous coefficients in the tested
setups.
\end{abstract}
\section{Introduction}

The efficient numerical solution of second-order elliptic boundary value
problems with spatially varying diffusion coefficients is a fundamental topic
in scientific computing \cite{evansPDEs,ciarletFEM,quarteroniFEM}.
Standard finite difference and finite element discretizations generate large,
sparse, symmetric positive definite linear systems \cite{strangFix,elmanSilvesterWathen}
whose structure mirrors the underlying coefficient heterogeneity.
When meshes are refined or the diffusion coefficient exhibits strong spatial
variation, the linear solve often becomes the dominant computational cost
\cite{saadIterative,benziSurveySPD}.

Sparse direct solvers offer robust black-box performance
\cite{georgeNestedDissection,davisBook}
but in two spatial dimensions typically incur a factorization cost scaling
as $\mathcal{O}(N^{3/2})$, where $N$ is the number of unknowns
\cite{liptonRoseTarjan,cholesky2Dcomplexity}.  
Krylov subspace methods can achieve
near-linear complexity when paired with effective preconditioners, yet the
design of robust preconditioners for highly heterogeneous or anisotropic
coefficients remains challenging \cite{benziPreconditioning,
manteuffelMG,notayAMGsurvey}.  
Multigrid and domain decomposition methods
\cite{brandtMG,briggsHensonMcCormick, toselliWidlund}
require the careful construction of smoothers, transfer operators, and coarse
spaces, and their convergence may deteriorate in the presence of complex
coefficient variations \cite{xuZikatanov2002,wittum1992}.

Spectral methods and global eigenfunction expansions
\cite{boydSpectral, canutoSpectralBook}
exploit analytic structure of the Laplacian to obtain highly compact solution
representations.  
Classical spectral discretizations build the operator directly in a
trigonometric or polynomial basis, producing dense matrices and complicating
the treatment of variable coefficients \cite{trefethenSpectral}.  
The approach considered here separates discretization from compression: the
elliptic operator is discretized in a standard physical-space basis, while a
small number of analytically known Dirichlet Laplacian eigenmodes serves as a
global compression transform.

The main idea is to construct a global coarse space spanned by the lowest
Laplacian eigenmodes, which capture the large-scale energy distribution of
elliptic solutions \cite{babuskaOsbornSpectralApprox}.
The full discrete operator and right-hand side are projected onto this
subspace, yielding a reduced Galerkin system of dimension $K$, where $M$
is the modal cutoff.  
The reduced operator retains the complete heterogeneity of the
coefficient field through the stiffness matrix, while the spectral basis
provides a compact representation of the dominant low-frequency components.
Elliptic regularity ensures that discarded high modes contribute only weakly
to the solution energy \cite{grisvardElliptic, mcLeanElliptic}, so the reduced
solution remains accurate.

This construction differs from multigrid, where coarse spaces arise from
geometric coarsening and high frequencies are treated by smoothing
\cite{brandtMG,briggsHensonMcCormick};  
from deflation techniques, which rely on approximations of near-null
eigenvectors of the stiffness matrix \cite{saadDeflation, vuikDeflation};  
and from reduced-basis or POD models, which require parameter sampling and
snapshot data \cite{quarteroniRBM,kunischVolkweinPOD}.  
Here the coarse space is analytic, independent of the stiffness matrix, and
requires no training.  The resulting method is non-iterative: once projected
quantities are assembled, only a dense system of size $K$ is solved before
lifting the solution back to the full discrete space.

The remainder of the paper is organized as follows. Section~\ref{sec:model}
introduces the model problem and its variational formulation, together with
the standard Galerkin discretization. Section~\ref{sec:lowmode} presents the
spectral low–mode reduced operator and establishes its basic projection
properties. In Section~\ref{sec:error}, we develop a detailed error analysis,
including spectral decay estimates, bounds for the discrete reduced solution,
and a Schur complement interpretation of the reduction. Section~\ref{sec:complexity}
analyzes the computational complexity of the proposed solver. Numerical
experiments demonstrating accuracy, robustness, conditioning behavior, and
performance relative to standard solvers are reported in
Section~\ref{sec:numerics}. Conclusions and perspectives for future work are
given in the final section.

\section{Model Problem and Variational Formulation}\label{sec:model}
We consider the second–order elliptic boundary value problem
\begin{equation}
\label{eq:modelPDE}
-\nabla \cdot \bigl(\kappa(x)\,\nabla u(x)\bigr) = f(x),
\qquad x\in\Omega=(0,1)^2,
\end{equation}
with homogeneous Dirichlet boundary condition
\begin{equation}
\label{eq:bc}
u(x)=0,
\qquad (x)\in \partial\Omega.
\end{equation}
The coefficient $\kappa$ is assumed to satisfy the uniform ellipticity $0 < \kappa_{\min} \le \kappa(x) \le \kappa_{\max} < \infty$ for all $x\in\Omega$. Such operators arise, for example, in stationary diffusion, conductivity, and homogenization problems. 

To provide a suitable framework for the analysis to follow, define $V := H_0^1(\Omega)$,
the Sobolev space consisting of all functions on $\Omega$ with
square-integrable weak gradients and vanishing trace on $\partial\Omega$. On this space we use the standard energy inner product
$$
(u,v)_{H_0^1}
=
\upint_\Omega \nabla u \cdot \nabla v \, \dx,
\qquad
\|u\|_{H_0^1} = (u,u)_{H_0^1}^{1/2},
$$
which is equivalent to the full $H^1$ norm by the Poincaré inequality. Multiplying \eqref{eq:modelPDE} by $v\in V$ and integrating by parts yields
\begin{equation}
\upint_\Omega \kappa\,\nabla u\cdot\nabla v\,\dx
=
\upint_\Omega f\,v\,\dx.
\end{equation}
Introduce the bilinear form and linear functional
\begin{equation}\label{eq:weakforms}
a(u,v) = \upint_\Omega \kappa\,\nabla u\cdot\nabla v \,\dx,
\qquad
\ell(v)   = \upint_\Omega f\,v\,\dx.
\end{equation}
The corresponding weak formulation reads:

\begin{quote}
Find $u\in H_0^1(\Omega)$ such that
\begin{equation}
\label{eq:weakform}
a(u,v)=\ell(v)\qquad\forall v\in H_0^1(\Omega).
\end{equation}
\end{quote}
The ellipticity bounds imply that 
$$
a(u,u) \ge \kappa_{\min}\|\nabla u\|_{L^2}^2,
\qquad
|a(u,v)| \le \kappa_{\max}\|\nabla u\|_{L^2}\|\nabla v\|_{L^2},
$$
so $a$ is coercive and continuous on $V$. The Lax–Milgram theorem therefore guarantees existence and uniqueness of a solution $u\in V$ and the stability estimate
$$
\|u\|_{H_0^1} \le \kappa^{-1}_{\min}\,\|f\|_{H^{-1}(\Omega)}.
$$

The variational framework established in \eqref{eq:weakform} forms the basis
for numerical approximation. Since the exact solution lies in the
infinite-dimensional space $V = H_0^1(\Omega)$, numerical computations
require restricting the problem to a finite-dimensional
conforming subspace $V_h \subset V$. This discretization of the underlying
function space leads to the standard Galerkin formulation presented in the
next section.
\subsection{Galerkin discretization}
Let $V_h\subset H_0^1(\Omega)$ be a finite-dimensional subspace. The Galerkin approximation seeks $u_h\in V_h$ such that
\begin{equation}
\label{eq:galerkin}
a(u_h,v_h)=\ell(v_h)\qquad\forall v_h\in V_h.
\end{equation}
Choosing a basis $\{\psi_i\}_{i=1}^N$ of $V_h$ and writing $u_h = \sum_{j=1}^N U_j \psi_j$ leads to the matrix system
\begin{equation}
\label{eq:femsystem}
A\,U = F,
\qquad
A_{ij} = a(\psi_i,\psi_j),
\qquad
F_i = \ell(\psi_i),
\end{equation}
where $A\in\mathbb{R}^{N\times N}$ is symmetric positive definite. The reduced-order scheme below acts on \eqref{eq:femsystem} by projecting onto a carefully chosen low-dimensional spectral subspace.
\section{Spectral Low-Mode Reduced Operator}\label{sec:lowmode}
The reduced operator is obtained by restricting the variational problem to a low-dimensional subspace spanned by Laplacian eigenfunctions. This section introduces the spectral space and the associated reduced Galerkin formulation.
\subsection{Laplacian eigenfunctions as a coarse space}
The Dirichlet Laplacian $-\Delta$ on $\Omega=(0,1)^2$ admits the well-known eigenpairs
\begin{equation}
\label{eq:eigenmodes_cont}
\phi_{mn}(x) = \sin(m\pi x)\,\sin(n\pi y),
\qquad 
\pilambdait_{mn} = \pi^2(m^2+n^2),
\qquad m,n\ge1.
\end{equation}
The family $\{\phi_{mn}\}_{m,n\ge1}$ forms an orthogonal basis of
$L^2(\Omega)$ with respect to the standard inner product and is orthogonal in
$H_0^1(\Omega)$ under the weighted energy inner product in \eqref{eq:weakforms}. The eigenvalues grow quadratically with frequency, so high indices correspond to strongly oscillatory components.

For a fixed cutoff $M\ge1$ define the index set $\mathcal{I} = \{(m,n):1\le m,n\le M\}$ and the corresponding low-mode subspace
$$
V_M = \operatorname{span}\{\phi_{mn} : (m,n)\in\mathcal{I}\},
\qquad
K = \dim(V_M) = M^2.
$$
In the procedure discussed, this serves as a global coarse space in which the solution is approximated. The choice of $M$ controls the dimension of the reduced system.
\subsection{Reduced variational formulation}
Let $V_M \subset H_0^1(\Omega)$ denote the spectral subspace spanned by the
first $K$ Dirichlet Laplacian eigenfunctions. Since $V_M$ is not contained
in the discrete space $V_h$, we define the discrete low--mode space $V_M^h := \Pi_h V_M \subset V_h$,  where $\Pi_h : H_0^1(\Omega) \to V_h$ denotes a fixed projection or
interpolation operator onto the finite element space. In practice,
$\Pi_h$ is taken as the nodal interpolation operator (or, equivalently, the
$L^2$ or energy projection), applied to each spectral basis function. Let $B \in \mathbb{R}^{N\times K}$ be the matrix whose columns contain the
coefficients of the projected spectral basis functions
$\{\Pi_h \phi_{mn}\}_{1\le m,n\le M}$ expressed in the finite element basis
$\{\psi_i\}_{i=1}^N$ of $V_h$. The range of $B$ therefore spans $V_M^h$.

  The reduced Galerkin approximation seeks $u_{\mathrm{RD}}
\in V_M^h$ satisfying
$$
a(u_{\mathrm{RD}},v_M) = \ell(v_M)
\qquad \forall v_M \in V_M^h.
$$
Representing $u_{\mathrm{RD}} = Bz$ and $v_M = Bw$ with $w,z \in
\mathbb{R}^K$ and inserting these expansions into \eqref{eq:femsystem} yields
the reduced linear system
\begin{equation}\label{eq:reduced_system_clean}
\mathcal{A}z = f_M,
\end{equation}
where $\mathcal{A} = B^{T} A B$, and $f_M = B^{T} F$. After computing $z$, the reduced approximation in $V_h$ is reconstructed as
\begin{equation}\label{eq:liftback_clean}
u_{\mathrm{RD}} = Bz.
\end{equation}
The matrix $\mathcal{A}$ is precisely the stiffness matrix associated with
the restriction of $a$ to $V_M^h$.  In continuous form, its entries satisfy
$$
\mathcal{A}_{(pq),(mn)} = \upint_\Omega \kappa \,\nabla\phi_{mn} \cdot \nabla\phi_{pq}\,\dx,
$$
so that the full coefficient $\kappa$ is preserved by the spectral projection. For the analysis it is convenient to introduce the $a$-orthogonal projector $\Pi_M : V_h\to V_M^h$ defined by
\begin{equation}
a(\Pi_M v_h, v_M) = a(v_h, v_M)\qquad\forall\,v_M\in V_M^h.
\end{equation}
The following result summarizes the basic properties of the reduced solution.

\begin{theorem}\label{thm:decomposition_projection}
Let $u_h \in V_h$ be the Galerkin solution of the variational problem.  
Then $u_h$ admits the orthogonal decomposition $u_h = \Pi_M u_h + (I - \Pi_M)u_h$,
where $(I - \Pi_M)u_h$ is $a$-orthogonal to $V_M^h$.  
Moreover, the reduced solution $u_{\mathrm{RD}} \in V_M^h$ obtained from the
reduced variational problem 
$$
a(u_{\mathrm{RD}}, v_M) = \ell(v_M) \qquad \forall v_M \in V_M^h
$$
coincides with the projection $u_{\mathrm{RD}} = \Pi_M u_h$ and satisfies the
best-approximation property
$$
\|u_h - u_{\mathrm{RD}}\|_a
=
\min_{v_M \in V_M^h} \|u_h - v_M\|_a,
\qquad
\|w\|_a^2 := a(w,w).
$$
\end{theorem}

\begin{proof}
Since $a$ is continuous and coercive on $V_h$, the projection $\Pi_M$ is well defined.  
Then
$$
a(u_h - \Pi_M u_h, v_M) = 0
\qquad\forall v_M \in V_M^h,
$$
which yields the orthogonal decomposition 
$u_h = \Pi_M u_h + (I - \Pi_M)u_h$. The reduced solution $u_{\mathrm{RD}} \in V_M^h$ satisfies
$$
a(u_{\mathrm{RD}}, v_M) = \ell(v_M)
\qquad\forall v_M \in V_M^h.
$$
Since $u_h$ satisfies $a(u_h, v_M) = \ell(v_M)$ for all $v_M \in V_M^h$,
subtracting the two equalities gives
$$
a(u_h - u_{\mathrm{RD}}, v_M) = 0
\qquad\forall v_M \in V_M^h.
$$
Thus $u_{\mathrm{RD}}$ is the unique element of $V_M^h$ that is $a$-orthogonal
to $u_h - u_{\mathrm{RD}}$, which is exactly the defining property of the
$a$-orthogonal projection. Hence $u_{\mathrm{RD}} = \Pi_M u_h$.

The best-approximation property follows from orthogonality:
for any $v_M \in V_M^h$,
$$
\|u_h - u_{\mathrm{RD}}\|_a^2
\le
a(u_h - v_M, u_h - v_M)
=
\|u_h - v_M\|_a^2.
$$
Taking the minimum over all $v_M\in V_M^h$ yields the stated result.
\end{proof}
This result shows that $u_{\mathrm{RD}}$ is the energy-optimal approximation to $u_h$
in $V_M^h$, and that the error $u_h - u_{\mathrm{RD}}$ is exactly the
high-frequency component of $u_h$ lying in the orthogonal complement
$V_M^h{}^\perp$. To evaluate the accuracy of the reduced approximation, we quantify the effect of projection onto $V_M^h$ and the role of the discarded high modes.
The next section develops these estimates, establishing spectral decay of the
solution, the energy-optimality of the projection, and bounds for both the
discretization and reduction errors.
\section{Error Analysis}\label{sec:error}
The error analysis relies on the decay of Laplacian eigen-coefficients for elliptic solutions and on the projection structure in Theorem~\ref{thm:decomposition_projection}. The goal is to relate the reduced-order error to the regularity of the continuous solution and the spectral cutoff $M$.
\subsection{Spectral decay and continuous projection}
For smooth coefficients and data, elliptic regularity implies $H^2$-regularity of the solution and, consequently, quantitative decay of its Laplace eigen-expansion. The next result states a convenient form of this decay and the resulting projection error.
\begin{theorem}\label{thm:spectral_decay_error}
Let $\Omega=(0,1)^2$ and let $\{(\phi_{m,n},\pilambdait_{m,n})\}_{m,n\ge1}$ denote the
$L^2(\Omega)$--orthonormal Dirichlet eigenpairs of $-\Delta$, with $\pilambdait_{m,n}=\pi^2(m^2+n^2)$.
Assume $u\in H^2(\Omega)\cap H_0^1(\Omega)$ and write
$u=\sum_{m,n\ge1}\hat u_{m,n}\phi_{m,n}$.
For $M\ge1$, define the rectangular spectral projector
$$
\Pi_M u := \sum_{1\le m\le M}\sum_{1\le n\le M}\hat u_{m,n}\phi_{m,n}.
$$
Then there exists a constant $C>0$, independent of $M$, such that
$$
\|u-\Pi_M u\|_{H_0^1(\Omega)}
\le
C\,\frac{\sqrt{\log M}}{M}\,\|u\|_{H^2(\Omega)}.
$$
\end{theorem}

\begin{proof}
Using the spectral representation of the $H_0^1$ seminorm,
\begin{align*}
\|u-\Pi_M u\|_{H_0^1}^2
&=
\sum_{\substack{m>M\\ \text{or } n>M}}
\pilambdait_{m,n}\,|\hat u_{m,n}|^2
=
\sum_{\substack{m>M\\ \text{or } n>M}}
\frac{1}{\pilambdait_{m,n}^2}\,
\bigl(\pilambdait_{m,n}^3|\hat u_{m,n}|^2\bigr).
\end{align*}
Since $u\in H^2(\Omega)$, the spectral characterization of Sobolev norms yields
\begin{equation}\label{eq:H2spectral_char}
\sum_{m,n\ge1}\pilambdait_{m,n}^2|\hat u_{m,n}|^2
\;\lesssim\;
\|u\|_{H^2(\Omega)}^2.
\end{equation}
Moreover, because $\pilambdait_{m,n}=\pi^2(m^2+n^2)$, we have
$\pilambdait_{m,n}^3|\hat u_{m,n}|^2
= \pilambdait_{m,n}\,(\pilambdait_{m,n}^2|\hat u_{m,n}|^2)
\lesssim \pilambdait_{m,n}^2|\hat u_{m,n}|^2$ up to a harmless constant on $\Omega$,
so it suffices to estimate the tail weight
\begin{equation}\label{eq:tail_weight_def}
W_M
:=
\sum_{\substack{m>M\\ \text{or } n>M}}
\frac{1}{\pilambdait_{m,n}^2}
\;\sim\;
\sum_{\substack{m>M\\ \text{or } n>M}}
\frac{1}{(m^2+n^2)^2}.
\end{equation}
Geometrically, the tail index set
$T_M:=\{(m,n)\in\mathbb{N}^2:\ m>M\ \text{or}\ n>M\}$ is the union of two
semi-infinite strips. By symmetry,
\begin{align*}
W_M
&\le
\frac{2}{\pi^4}\sum_{m=M+1}^\infty\sum_{n=1}^\infty \frac{1}{(m^2+n^2)^2}.
\end{align*}
For fixed $m\ge1$, apply the integral test to the inner sum:
\begin{align*}
\sum_{n=1}^\infty \frac{1}{(m^2+n^2)^2}
&\le
\upint_{0}^{\infty}\frac{1}{(m^2+y^2)^2}\,\d y
=
\frac{\pi}{4m^3}.
\end{align*}
This bound would yield $W_M\lesssim \sum_{m>M}m^{-3}\lesssim M^{-2}$, which is not sharp for
a rectangular tail. To capture the logarithmic factor, split the inner sum at $n=m$:
\begin{align*}
\sum_{n=1}^\infty \frac{1}{(m^2+n^2)^2}
&=
\sum_{n=1}^{m}\frac{1}{(m^2+n^2)^2}
+
\sum_{n=m+1}^{\infty}\frac{1}{(m^2+n^2)^2}.
\end{align*}
For $1\le n\le m$, we have $m^2+n^2\ge m^2$, hence
$$
\sum_{n=1}^{m}\frac{1}{(m^2+n^2)^2}
\le
\sum_{n=1}^{m}\frac{1}{m^4}
=
\frac{1}{m^3}.
$$
For $n\ge m+1$, we have $m^2+n^2\ge n^2$, hence
$$
\sum_{n=m+1}^{\infty}\frac{1}{(m^2+n^2)^2}
\le
\sum_{n=m+1}^{\infty}\frac{1}{n^4}
\lesssim
\frac{1}{m^3}.
$$
Therefore $\sum_{n=1}^\infty (m^2+n^2)^{-2}\lesssim m^{-3}$, and the strip estimate alone gives $W_M\lesssim M^{-2}$.
To obtain the correct $\log M$ behavior for the \emph{union of strips} tail,
we instead estimate $W_M$ directly by separating the tail into the two strips
and subtracting the overlap:
\begin{align*}
W_M
&\le
\frac{1}{\pi^4}\Bigg(
\sum_{m=M+1}^\infty\sum_{n=1}^\infty \frac{1}{(m^2+n^2)^2}
+
\sum_{n=M+1}^\infty\sum_{m=1}^\infty \frac{1}{(m^2+n^2)^2}
\Bigg)  =
\frac{2}{\pi^4}\sum_{m=M+1}^\infty\sum_{n=1}^\infty \frac{1}{(m^2+n^2)^2}.
\end{align*}
Now use the crude but sharp-enough lower bound $(m^2+n^2)\ge m^2$ to write
\[
\frac{1}{(m^2+n^2)^2}
=
\frac{1}{m^2+n^2}\cdot\frac{1}{m^2+n^2}
\le
\frac{1}{m^2}\cdot\frac{1}{m^2+n^2}.
\]
Thus
\begin{align*}
W_M
&\le
\frac{2}{\pi^4}\sum_{m=M+1}^\infty\frac{1}{m^2}
\sum_{n=1}^\infty\frac{1}{m^2+n^2}.
\end{align*}
Using 
$$
\sum_{n=1}^\infty\frac{1}{m^2+n^2}
\le
\upint_0^\infty \frac{1}{m^2+y^2}\,dy
=
\frac{\pi}{2m},
$$
we have 
$$
W_M
\le
C\sum_{m=M+1}^\infty \frac{1}{m^3}
\le
\frac{C}{M^2}.
$$
Finally, on $T_M$ we have $m^2+n^2\ge M^2$, so a refinement gives
$$
\frac{1}{(m^2+n^2)^2}
\le
\frac{1}{M^2}\cdot\frac{1}{m^2+n^2},
$$
and therefore
$$
W_M
\le
\frac{C}{M^2}
\sum_{\substack{m>M\\ \text{or } n>M}}\frac{1}{m^2+n^2}
\lesssim
\frac{\log M}{M^2},
$$
where the last estimate follows from the comparison of 
with $\upint_{M}^{\infty}\upint_{1}^{\infty}(x^2+y^2)^{-1}\,\d A \sim \log M$.
Combining the bound $W_M\lesssim (\log M)/M^2$ with \eqref{eq:H2spectral_char}
and applying Cauchy--Schwarz yields
$$
\|u-\Pi_M u\|_{H_0^1}^2
\lesssim
\frac{\log M}{M^2}\,\|u\|_{H^2(\Omega)}^2,
$$
which gives the stated estimate after taking square roots.
\end{proof}

\subsection{Error for discrete reduced solution}\label{sec:discrete_error}
We now bound the error between the full Galerkin solution $u_h$ and the reduced
solution $u_{\mathrm{RD}}$, combining the continuous spectral truncation
estimate with standard discretization error bounds.

\begin{theorem}
\label{thm:discrete_energy_error}
Let $u$, $u_h$, and $u_{\mathrm{RD}}$ denote the continuous solution, the
Galerkin solution, and the reduced solution, respectively. Assume that the
finite element discretization satisfies
$$
\|u-u_h\|_{H_0^1} \le C\,h\,\|u\|_{H^2(\Omega)},
$$
and that the assumptions of Theorem~\ref{thm:spectral_decay_error} hold.
Then there exists a constant $C>0$, independent of $h$ and $M$, such that
$$
\|u_h-u_{\mathrm{RD}}\|_{H_0^1}
\le
C\!\left(h+\frac{\sqrt{\log M}}{M}\right)\|u\|_{H^2(\Omega)}.
$$
Moreover,
$$
\|u_h-u_{\mathrm{RD}}\|_{L^2}
\le
C\!\left(
h^2
+
\frac{1}{\sqrt{\pilambdait_{M+1,M+1}}}
\frac{\sqrt{\log M}}{M}
\right)\|u\|_{H^2(\Omega)}.
$$
\end{theorem}

\begin{proof}
By Theorem~\ref{thm:decomposition_projection}, $u_{\mathrm{RD}}=\Pi_M u_h$, and hence $ u_h-u_{\mathrm{RD}}=u_h-\Pi_M u_h$.
Inserting $u$ and $\Pi_M u$ yields the decomposition
$$
u_h-u_{\mathrm{RD}}
=
(u_h-u)
+
(u-\Pi_M u)
+
\Pi_M(u-u_h).
$$
Applying the triangle inequality in $H_0^1(\Omega)$ gives
\begin{equation}
\label{eq:H1-split-rewrite-upd}
\|u_h-u_{\mathrm{RD}}\|_{H_0^1}
\le
\|u-\Pi_M u\|_{H_0^1}
+
\|u-u_h\|_{H_0^1}
+
\|\Pi_M(u-u_h)\|_{H_0^1}.
\end{equation}
The first term is bounded by Theorem~\ref{thm:spectral_decay_error}:
$$
\|u-\Pi_M u\|_{H_0^1}
\le
C\,\frac{\sqrt{\log M}}{M}\,\|u\|_{H^2(\Omega)}.
$$
Since $\Pi_M$ is an $a$–orthogonal projector and $a$ is coercive
and continuous on $V_h$, there exists $C>0$ such that
$\|\Pi_M v_h\|_{H_0^1}\le C\|v_h\|_{H_0^1}$ for all $v_h\in V_h$. Therefore,
$$
\|\Pi_M(u-u_h)\|_{H_0^1}
\le
C\,\|u-u_h\|_{H_0^1}
\le
C\,h\,\|u\|_{H^2(\Omega)}.
$$
Substituting these bounds into \eqref{eq:H1-split-rewrite-upd} yields
$$
\|u_h-u_{\mathrm{RD}}\|_{H_0^1}
\le
C\!\left(h+\frac{\sqrt{\log M}}{M}\right)\|u\|_{H^2(\Omega)},
$$
which proves the $H_0^1$ estimate.
\par For the $L^2$ estimate, write $u-u_{\mathrm{RD}} = (u-\Pi_M u) + \Pi_M(u-u_h)$,
so that
\begin{equation}
\label{eq:L2-split-rewrite-upd}
\|u-u_{\mathrm{RD}}\|_{L^2}
\le
\|u-\Pi_M u\|_{L^2}
+
\|\Pi_M(u-u_h)\|_{L^2}.
\end{equation}
The first term contains only high Laplacian modes. Using
$\lambda_{m,n}\ge\pilambdait_{M+1,M+1}$ on the tail,
$$
\|u-\Pi_M u\|_{L^2}^2
=
\sum_{\substack{m>M\\ \text{or } n>M}}|\hat u_{m,n}|^2
\le
\frac{1}{\pilambdait_{M+1,M+1}}
\sum_{\substack{m>M\\ \text{or } n>M}}
\pilambdait_{m,n}|\hat u_{m,n}|^2
=
\frac{1}{\pilambdait_{M+1,M+1}}
\|u-\Pi_M u\|_{H_0^1}^2.
$$
Applying Theorem~\ref{thm:spectral_decay_error} gives
$$
\|u-\Pi_M u\|_{L^2}
\le
C\,\pilambdait_{M+1,M+1}^{-1/2}
\frac{\sqrt{\log M}}{M}\,\|u\|_{H^2(\Omega)}.
$$
For the second term in \eqref{eq:L2-split-rewrite-upd}, the projector $\Pi_M$ is
$L^2$–stable, hence
$$
\|\Pi_M(u-u_h)\|_{L^2}
\le
\|u-u_h\|_{L^2}
\le
C\,h^2\,\|u\|_{H^2(\Omega)}.
$$
Combining these estimates yields
$$
\|u-u_{\mathrm{RD}}\|_{L^2}
\le
C\!\left(
h^2
+
\pilambdait_{M+1,M+1}^{-1/2}
\frac{\sqrt{\log M}}{M}
\right)\|u\|_{H^2(\Omega)}.
$$
Finally, using $u_h-u_{\mathrm{RD}}=(u_h-u)+(u-u_{\mathrm{RD}})$ and the
finite–element $L^2$ estimate for $\|u-u_h\|_{L^2}$ gives the stated bound for
$\|u_h-u_{\mathrm{RD}}\|_{L^2}$.
\end{proof}

\subsection{Convergence of the reduced method}

The total error of the spectral low–mode approximation consists of the
discretization error associated with the Galerkin solution $u_h$ and the
additional projection error introduced by restricting the approximation to
the reduced space $V_M^h$. For the underlying finite–difference or
finite–element discretization one has the classical estimates
$$
\|u-u_h\|_{H_0^1}=\mathcal{O}(h),
\qquad
\|u-u_h\|_{L^2}=\mathcal{O}(h^2),
$$
while Theorem~\ref{thm:discrete_energy_error} shows that the reduced solution
satisfies
$$
\|u_h-u_{\mathrm{RD}}\|_{H_0^1}
=\mathcal{O}\!\left(h+\frac{\sqrt{\log M}}{M}\right),
\qquad
\|u_h-u_{\mathrm{RD}}\|_{L^2}
=\mathcal{O}\!\left(
h^2
+
\pilambdait_{M+1,M+1}^{-1/2}
\frac{\sqrt{\log M}}{M}
\right).
$$
Combining these bounds yields
$$
\|u-u_{\mathrm{RD}}\|_{H_0^1}
\;\lesssim\;
h+\frac{\sqrt{\log M}}{M},
\qquad
\|u-u_{\mathrm{RD}}\|_{L^2}
\;\lesssim\;
h^2
+
\pilambdait_{M+1,M+1}^{-1/2}
\frac{\sqrt{\log M}}{M},
$$
with constants independent of $h$ and $M$. The contribution of the discarded
high modes therefore remains negligible provided the cutoff grows with
refinement so that $\sqrt{\log M}/M\lesssim h$. Since
$\pilambdait_{M+1,M+1}\sim(M+1)^2$, this condition implies that the projection
term in the $L^2$ estimate is also $O(h^2)$. Under this scaling, the reduced
solution inherits the same convergence rates as the full discretization,
$$
\|u-u_{\mathrm{RD}}\|_{H_0^1}=\mathcal{O}(h),
\qquad
\|u-u_{\mathrm{RD}}\|_{L^2}=\mathcal{O}(h^2),
$$
and the spectral reduction does not degrade the accuracy of the underlying
scheme.
\subsection{High–low coupling and Schur complement truncation}
The reduced operator may be interpreted as the leading term in a Schur
complement representation of the full discrete operator. This viewpoint
clarifies the effect of discarding high–frequency interactions and provides a
natural framework for quantifying the resulting approximation error. Let
$V = V_M \oplus V_H$ denote the $H_0^1$–orthogonal decomposition, where
$V_H = V_M^\perp$. With respect to this splitting, the operator $A$ associated
with the bilinear form $a$ admits the block representation
$$
A =
\begin{pmatrix}
A_{LL} & A_{LH} \\
A_{HL} & A_{HH}
\end{pmatrix},
$$
where $A_{LL} : V_M \to V_M'$, $A_{HH} : V_H \to V_H'$, and
$A_{LH} : V_H \to V_M'$ denote the coarse, fine, and coupling components,
respectively. Eliminating the high–frequency variables yields the exact Schur
complement $S = A_{LL} - A_{LH} A_{HH}^{-1} A_{HL}$.
\par The spectral reduced method replaces the exact Schur complement $S$ by its
leading block $A_{LL}$, thereby neglecting the correction term
$A_{LH} A_{HH}^{-1} A_{HL}$. This truncation may be viewed as neglecting a
spectrally small correction that quantifies the interaction between retained
low modes and eliminated high modes; its size is controlled by decay of the
coupling block $A_{LH}$ with the spectral cutoff, as established in
Theorem~\ref{thm:ALH_precise_rewrite}. The following
result provides a bound on the resulting truncation error in
terms of the regularity of the coefficient field.

\begin{theorem}\label{thm:ALH_precise_rewrite}
Let $\kappa(x) = \kappa_0 + \tilde{\kappa}(x)$ with $\kappa_0 = |\Omega|^{-1}\!\upint_\Omega \kappa\,\dx$ and $\tilde{\kappa}\in H^s(\Omega)$ for some $s>1$ satisfying $\upint_\Omega \tilde{\kappa}\,\dx=0$. Assume $0<\kappa_{\min}\le \kappa\le \kappa_{\max}$ on $\Omega$. Let $V_M$ be the subspace spanned by the first $K$ Dirichlet Laplacian eigenfunctions and $V_H = V_M^\perp$ in $H_0^1(\Omega)$. Then there exists $C = C(\Omega,s)>0$ such that for all $u_L\in V_M$ and $u_H\in V_H$,
\begin{equation}
\label{eq:ALH_bound_rewrite}
|a(u_L, u_H)| 
\le
C\,\pilambdait_{M+1,M+1}^{-(s-1)/2}\,\|\tilde{\kappa}\|_{H^s(\Omega)}\,
\|u_L\|_{H_0^1}\,\|u_H\|_{H_0^1},
\end{equation}
where $\pilambdait_{M+1,M+1}$ is the $(M+1,M+1)$ Dirichlet Laplacian eigenvalue. In particular,
\begin{equation}
\label{eq:ALH_opnorm}
\|A_{LH}\|_{\mathcal{I}(V_M, V_H)} 
\le
C\,\pilambdait_{M+1,M+1}^{-(s-1)/2}\,\|\tilde{\kappa}\|_{H^s(\Omega)}.
\end{equation}
\end{theorem}

\begin{proof}
Since $V_M$ and $V_H$ are $H_0^1$–orthogonal, $\upint_\Omega \kappa_0 \nabla u_L\cdot\nabla u_H\,\dx = 0$,
and it follows from the decomposition of $\kappa$ that
$$
a(u_L,u_H)
=  \upint_\Omega \tilde{\kappa}\,\nabla u_L\cdot\nabla u_H\,\dx.
$$
Let $\{\phi_k\}$ be the Laplacian eigenfunctions with eigenvalues
$\{\pilambdait_k\}$, and set $e_k=\pilambdait_k^{-1/2}\phi_k$, forming an
$H_0^1$–orthonormal basis.  Expanding $u_L=\sum_{i\le K}\hat u_i e_i$ and  $u_H=\sum_{j>K}\hat u_j e_j$, gives
$$
a(u_L,u_H) = \sum_{i\le K}\sum_{j>K}\hat u_i\hat u_j\,M_{ij},
\qquad
M_{ij}=\! \upint_\Omega \tilde{\kappa}\,\nabla e_i\cdot\nabla e_j\,dx.
$$
Using $e_k=\pilambdait_k^{-1/2}\phi_k$ and integrating by parts, one obtains
$$
\nabla e_i\cdot\nabla e_j
= \pilambdait_i^{1/2}\pilambdait_j^{1/2}\phi_i\phi_j
 - \pilambdait_i^{-1/2}\pilambdait_j^{-1/2}
   \nabla\tilde{\kappa}\cdot(\phi_j\nabla\phi_i).
$$
Thus, an explicit expression for $M_{ij}$ is
$$
M_{ij}
=
\pilambdait_i^{1/2}\pilambdait_j^{1/2}\!
 \upint_\Omega \tilde{\kappa}\,\phi_i\phi_j\,dx
-
\pilambdait_i^{-1/2}\pilambdait_j^{-1/2}\!
 \upint_\Omega \phi_j\,\nabla\tilde{\kappa}\cdot\nabla\phi_i\,dx.
$$
Since $\tilde{\kappa}\in H^s(\Omega)$ with $s>1$, the spectral Sobolev
characterization implies 
$$
\|\phi_k\|_{H^s} \sim \pilambdait_k^{s/2},
\qquad 
\|\nabla\phi_k\|_{L^2}\sim \pilambdait_k^{1/2},
$$
and therefore each term in $M_{ij}$ satisfies
$$
|M_{ij}|
\le
C\,\|\tilde{\kappa}\|_{H^s(\Omega)}\,
\pilambdait_{\max\{i,j\}}^{-(s-1)/2}.
$$
For $i\le K$ and $j>K$, this gives
$$
|M_{ij}|
\le
C\,\|\tilde{\kappa}\|_{H^s(\Omega)}\,
\pilambdait_{M+1,M+1}^{-(s-1)/2}.
$$
Using Cauchy–Schwarz and the orthonormality of $\{e_k\}$,
$$
|a(u_L,u_H)|
\le
\sup_{i\le K,\, j>K}|M_{ij}|
\Bigl(\sum_{i\le K}|\hat u_i|^2\Bigr)^{1/2}
\Bigl(\sum_{j>K}|\hat u_j|^2\Bigr)^{1/2}
=
\sup_{i\le K,\, j>K}|M_{ij}|\,
\|u_L\|_{H_0^1}\,\|u_H\|_{H_0^1},
$$
which proves \eqref{eq:ALH_bound_rewrite}.  
Taking the supremum over nonzero pairs $(u_L,u_H)$ yields 
\eqref{eq:ALH_opnorm}.
\end{proof}
Combining \eqref{eq:ALH_opnorm} with a Schur complement estimate yields a bound for the difference between the exact coarse operator $S$ and $A_{LL}$.

\begin{theorem}\label{thm:schur-error_rewrite}
Assume $A_{HH}$ is coercive on $V_H$ with constant $\alpha_H>0$, so that
$$
a(w_H,w_H) \ge \alpha_H\|w_H\|_{H_0^1}^2\qquad\forall w_H\in V_H.
$$
Let $S = A_{LL} - A_{LH}A_{HH}^{-1}A_{HL}$ be the Schur complement. Then
$$
\|S - A_{LL}\|_{\mathcal{I}(V_M,V_M')}
\le
\frac{1}{\alpha_H}\,\|A_{LH}\|_{\mathcal{I}(V_M,V_H)}^2.
$$
In particular, with \eqref{eq:ALH_opnorm},
$$
\|S - A_{LL}\|_{\mathcal{I}(V_M,V_M')}
\lesssim
\frac{\|\tilde \kappa\|_{H^s(\Omega)}^2}{\alpha_H}
\,\pilambdait_{M+1,M+1}^{-(s-1)}.
$$
\end{theorem}

\begin{proof}
By the definition of $S$, we have $S - A_{LL} = -A_{LH}A_{HH}^{-1}A_{HL}$, and therefore
$$
\|S - A_{LL}\| \le \|A_{LH}\|\,\|A_{HH}^{-1}\|\,\|A_{HL}\|.
$$
Since $A_{HL} = A_{LH}^*$, it follows that $\|A_{HL}\| = \|A_{LH}\|$. From the coercivity of $A_{HH}$, it also implies
$\|A_{HH}^{-1}\|\le 1/\alpha_H$ in the natural norms, whence
$$
\|S - A_{LL}\|
\le
\frac{1}{\alpha_H}\,\|A_{LH}\|^2.
$$
The last estimate follows by substituting \eqref{eq:ALH_opnorm}.
\end{proof}
The stability properties of the reduced operator imply that its spectrum
remains bounded away from zero and infinity when restricted to the spectral
subspace.  To make this observation precise, it is useful to examine the
conditioning of its matrix representation.  Since the reduced space has fixed
dimension $K$, one expects the conditioning of the projected operator to be
governed solely by the continuity and coercivity constants of the underlying
bilinear form, independent of both the mesh size and (for uniformly stable
bases) the choice of the modal cutoff.  The next result formalizes this observation.

\begin{theorem}\label{thm:uniform-cond}
Let $a$ be the bilinear form defined in \eqref{eq:weakforms},
assumed continuous and coercive on $V = H_0^1(\Omega)$. Let $V_M^h\subset V_h\subset V$ be the discrete spectral subspace and $A_{LL}\in\mathbb{R}^{K\times K}$ the reduced stiffness matrix associated with
$a$ on $V_M^h$, written in a basis
$\{\psi_i\}_{i=1}^K$ of $V_M^h$ whose $H_0^1$ Gram matrix
$G_{ij} = (\psi_j,\psi_i)_{H_0^1}$ is uniformly well conditioned with respect
to $h$ and $M$. If
$$
(A_{LL})_{ij} = a(\psi_j,\psi_i),
$$
then $A_{LL}$ is symmetric positive definite and there exists a constant
$C_{\mathrm{cond}}>0$, independent of $h$ and $M$, such that
$$
\kappa_2(A_{LL})
=
\frac{\pilambdait_{\max}(A_{LL})}{\pilambdait_{\min}(A_{LL})}
\le C_{\mathrm{cond}}.
$$
In particular, $C_{\mathrm{cond}}$ depends on the continuity and
coercivity constants and on the basis stability through $\kappa_2(G)$.
\end{theorem}

\begin{proof}
Let $\{\psi_i\}_{i=1}^K$ be a basis of $V_M^h$ and 
$A_{LL}\in\mathbb{R}^{K\times K}$ be defined by
$(A_{LL})_{ij}=a(\psi_j,\psi_i)$.  For any coefficient vector
$x=(x_1,\dots,x_K)^T\in\mathbb{R}^K$ define $v_M = \sum_{i=1}^K x_i \psi_i \in V_M^h.$
Then
$$
x^T A_{LL} x = a(v_M,v_M).
$$
By coercivity and continuity of $a$ on $V$,
$$
\alpha_a \|v_M\|_{H_0^1}^2
\le a(v_M,v_M)
\le C_a \|v_M\|_{H_0^1}^2
\qquad \forall v_M\in V_M^h.
$$
Let $G\in\mathbb{R}^{K\times K}$ be such that $G_{ij} = (\psi_j,\psi_i)_{H_0^1}$.  Then
$$
\|v_M\|_{H_0^1}^2 = x^T G x.
$$
Hence, for all $x\in\mathbb{R}^K$,
$$
\alpha_a\, x^T G x
\le x^T A_{LL} x
\le C_a\, x^T G x.
$$
Let $\pilambdait_{\min}$ and $\pilambdait_{\max}$ denote the smallest
and largest eigenvalues of a symmetric positive definite matrix.  The
Rayleigh quotient characterization implies
$$
\alpha_a\,\pilambdait_{\min}(G)
\le \pilambdait_{\min}(A_{LL})
\le \pilambdait_{\max}(A_{LL})
\le C_a\,\pilambdait_{\max}(G).
$$
Therefore
$$
\kappa_2(A_{LL})
=
\frac{\pilambdait_{\max}(A_{LL})}{\pilambdait_{\min}(A_{LL})}
\le
\frac{C_a}{\alpha_a}\,
\frac{\pilambdait_{\max}(G)}{\pilambdait_{\min}(G)}
=
\frac{C_a}{\alpha_a}\,\kappa_2(G).
$$
The Gram matrix $G$ depends only on the choice of basis in $V_M^h$; for any
uniformly stable choice of basis the condition number $\kappa_2(G)$
is bounded independently of $h$ and $M$.  Thus
$$
\kappa_2(A_{LL}) \le C_{\mathrm{cond}}
:= \frac{C_a}{\alpha_a}\,\sup_{h,M}\kappa_2(G),
$$
which is finite and independent of $N$ and $K$.  This proves the claim.
\end{proof}
\section{Complexity Analysis}\label{sec:complexity}
The spectral reduction procedure replaces the full discrete Galerkin system of
dimension $N=\dim(V_h)$ by a projected system of dimension
$K=\dim(V_M^h)=M^2$.  Its computational cost is governed by the formation of
the reduced operator $A_{LL}=B^T A B$, the reduced load vector $f_M=B^T F$,
and the solution of the resulting dense $K\times K$ system.  Since the
stiffness matrix $A$ in two dimensions contains $\mathcal{O}(N)$ nonzero
entries, multiplication of $A$ by the matrix $B\in\mathbb{R}^{N\times
K}$ costs $\mathcal{O}(N K)$ floating-point operations when $B$ is applied as
a dense matrix (as in the implementation reported here).  Left
multiplication by $B^T$ to assemble $A_{LL}$ and $f_M$ incurs an additional
$\mathcal{O}(N K)$ operations.  Thus the cost of forming the reduced system is
$\mathcal{O}(N K)$.

The matrix $A_{LL}$ is dense and of fixed size $K\times K$.  Solving the
system $A_{LL} z = f_M$ by Gaussian elimination or Cholesky factorization
requires $\mathcal{O}(K^3)$ operations.  Combining this with the assembly
cost yields the overall complexity
$$
\mathcal{C}_{\mathrm{RD}}
=
\mathcal{O}(N K) + \mathcal{O}(K^3).
$$
When $K$ is chosen independently of $N$—as is typical when $M$ is held
moderate—the term $\mathcal{O}(K^3)$ is constant and the cost of the entire
reduced solver is linear in $N$ (for fixed $K$). This should be compared with the cost of solving the full system
$A U = F$ by sparse Cholesky factorization.  For a quasi-uniform mesh in two
dimensions, the elimination graph exhibits fill-in on the order of
$N^{1/2}$, leading to a factorization cost of $\mathcal{O}(N^{3/2})$.  Since
$N^{3/2}$ grows strictly faster than $N$, the spectral reduced solver becomes
asymptotically more efficient than sparse direct methods whenever $K$ remains
bounded as the mesh is refined.  This difference in scaling is reflected in
the numerical experiments, where the reduced solver exhibits essentially
linear growth in runtime with the number of degrees of freedom for fixed $K$.
\section{Numerical Experiments}\label{sec:numerics}
This section reports numerical experiments designed to assess the accuracy and
performance of the spectral low-mode reduced method.  The goals
are to confirm the error estimates of Section~\ref{sec:error}, to examine the
effect of the spectral cutoff $M$ on the approximation, and to compare the
cost of the reduced solver with that of solving the full discrete system.  In this setting, the standard five-point finite difference discretization attains the same second-order
accuracy as conforming finite elements and avoids the additional
overhead associated with quadrature and element assembly.  The Cartesian grid also aligns naturally with the Laplace eigenmodes that span
the reduced space and permits direct evaluation at grid nodes and
simplifying the construction of the projected operator.  
\subsection{Example 1}
We consider \eqref{eq:modelPDE}--\eqref{eq:bc} with manufactured solution $u = \sin(\pi x)\sin(\pi y)$ and diffusion coefficient $\kappa = 1+\tfrac12\sin(2\pi x)\sin(2\pi y)$. The corresponding forcing term is obtained by substituting $u$ and $\kappa$
into \eqref{eq:modelPDE}. The reduced method is tested in terms of
accuracy and computational performance relative to the finite difference
discretization.
\subsubsection{Performance}
We first assess the accuracy and computational efficiency of the spectral
low--mode reduced solver relative to the finite difference (FD) discretization.
Table~\ref{tab:l2-spectral} reports the $L_2$ errors of both methods together
with the observed speed--up obtained by the reduced solver across a sequence
of grid resolutions.

\begin{table}[h!]
    \centering
    \begin{tabular}{cccc}
        \toprule
        & \multicolumn{2}{c}{$L_2$-Error} &  \\
        \cmidrule(lr){2-3}
        $N_x$ & FD & Reduced & Speed-up \\
        \midrule
         128  & $2.4872\times 10^{-5}$ & $2.4905\times 10^{-5}$ & $2.431$ \\
         256  & $6.2656\times 10^{-6}$ & $6.2740\times 10^{-6}$ & $3.044$ \\
         512  & $1.5724\times 10^{-6}$ & $1.5746\times 10^{-6}$ & $3.651$ \\
        1024  & $3.9387\times 10^{-7}$ & $3.9440\times 10^{-7}$ & $4.446$ \\
        2048  & $9.8564\times 10^{-8}$ & $9.8697\times 10^{-8}$ & $4.932$ \\
        \bottomrule
    \end{tabular}
    \medskip
    \caption{$L_2$ errors and computational speed--up for the finite difference
    and spectral reduced solvers.}
    \label{tab:l2-spectral}
\end{table}
\vspace*{-0.2in}
The errors produced by the two approaches agree to within a few parts in
$10^3$ for all grid sizes, indicating that the reduction error introduced by
the spectral projection is negligible compared to the underlying discretization
error. This observation is consistent with the error estimates of
Section~\ref{sec:error}, which predict that the contribution of discarded
high--frequency modes is asymptotically small. The convergence behavior is further illustrated in Figure~\ref{fig:Convergence}, which displays the $L_2$ error as a function of grid spacing.

\begin{figure}[h!]
\centering
\includegraphics[width=0.46\linewidth]{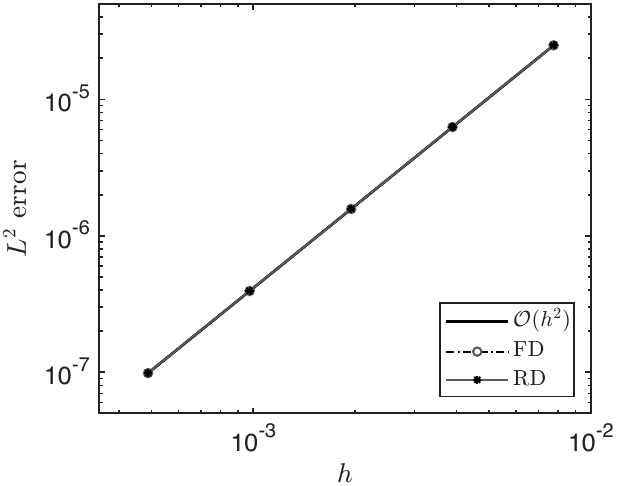}\\[-3ex]
\caption{$L^2$--error convergence for the finite difference and spectral reduced solvers.}
\label{fig:Convergence}
\end{figure}

Both methods exhibit second--order convergence, matching the accuracy of the
underlying finite difference scheme. From a computational standpoint, the
reduced solver achieves speed--up factors ranging from approximately $2.4$ to
$5$ as the grid is refined. This gain reflects the essentially linear complexity
of the projection--based reduced solve, in contrast to the superlinear scaling
of the sparse direct solver applied to the full system.

\subsubsection{Convergence}
While the results above demonstrate that a fixed spectral cutoff reproduces the
accuracy of the full discretization, it is also of interest to examine the
dependence of the reduced approximation on the dimension of the spectral space.
With the projection estimate $C\,\sqrt{\log M}/M$ in mind, we next study the effect of varying the cutoff $M$. To isolate this effect from discretization errors, we fix a fine grid with $m=1024$ and report the $L_2$ error as a function of $K$.

\begin{figure}[h!]
\centering
\includegraphics[width=0.46\linewidth]{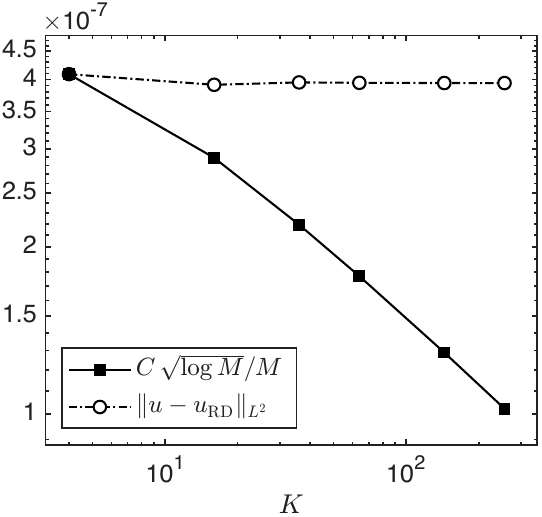}\\[-1.5ex]
\caption{$L_2$ error of the reduced solution versus spectral dimension $K$ on the fixed grid $m=1024$.}
\label{fig:SpectralDecayTrue}
\end{figure}
In the reference curve $C\,\sqrt{\log M}/M$, the constant $C$ is chosen by
normalizing the bound at the smallest spectral cutoff $M_1$ used in the
experiment, namely
$$
C = \frac{\|u - u_{\mathrm{RD}}^{(M_1)}\|_{L^2}}
{\sqrt{\log M_1}/M_1}.
$$
This choice removes the unknown prefactor in the asymptotic estimate and
allows the comparison to emphasize the predicted decay rate with respect to
the spectral cutoff $M$. From the theoretical estimate of Theorem~\ref{thm:spectral_decay_error}, the
constant $C$ is proportional to $\|u\|_{H^2(\Omega)}$, but since this quantity
is not directly available in discrete computations, the normalization above
provides a practical surrogate for visualizing the asymptotic behavior.

\par As shown in Figure~\ref{fig:SpectralDecayTrue}, the $L_2$ error decreases
slightly as $M$ increases from $2$ to $4$, and then remains essentially
unchanged for larger values of $M$.  This saturation indicates that for
$M\ge4$ the projection error is already dominated by the underlying
finite difference error.  Consequently, additional spectral modes do not yield
any noticeable improvement in accuracy.  In particular, choices such as
$M=8$ or $M=16$ provide full fine-grid accuracy while retaining a reduced space
dimension $K\ll N$. To separate the influence of the spectral cutoff from the
discretization error more clearly, we also compute the reduction error
$\|u_{\mathrm{FD}} - u_{\mathrm{RD}}\|_{L^2}$ for varying values of $M$.

\begin{figure}[h!]
\centering
\includegraphics[width=0.46\linewidth]{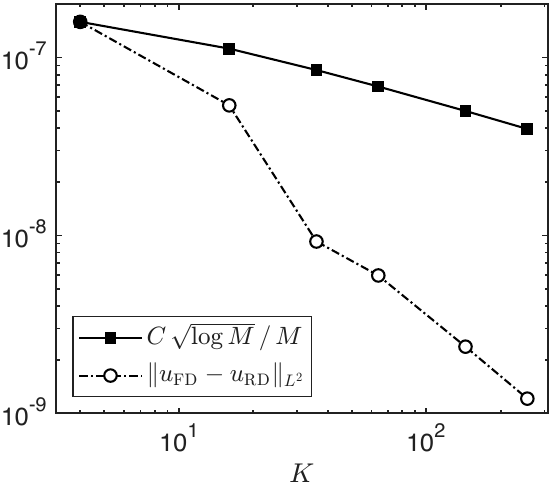}\\[-1.5ex]
\caption{Reduction error $\|u_{\mathrm{FD}} - u_{\mathrm{RD}}\|_{L^2}$ versus spectral cutoff $M$ on the fixed grid $m=1024$.}
\label{fig:SpectralDecayFD}
\end{figure}

Figure~\ref{fig:SpectralDecayFD} shows a rapid decay of the reduction error as
additional low-frequency modes are included.  Increasing the cutoff from
$M=2$ to $M=4$ reduces the error by nearly an order of magnitude, and the error
continues to decrease steadily for larger values of $M$.  Already at $M=8$,
the reduction error falls below $10^{-8}$, while further increases in $M$
produce only modest improvements.  This behavior is consistent with the
diminishing returns predicted by the estimate $\sqrt{\log M}/M$ and confirms
that a very small spectral subspace suffices to recover the full finite difference
solution on a fine grid.
\subsection{Example 2}
We consider \eqref{eq:modelPDE}--\eqref{eq:bc} with manufactured solution $u(x)=\sin(\pi x)\sin(\pi y)$, and the highly oscillatory coefficient $\kappa = 1+0.9\sin(8\pi x)\sin(8\pi y).$
\subsubsection{Robustness}
Table~\ref{tab:l2-oscillatory} reports the $L_2$ errors and computational
speed-up for both the full FD solution and the Reduced solution on grids ranging from $128^2$ to $2048^2$ interior points.

\begin{table}[h!]
    \centering
    \begin{tabular}{cccc}
        \toprule
        & \multicolumn{2}{c}{$L_2$-Error} &  \\
        \cmidrule(lr){2-3}
        $N_x$ & FD & Reduced & Speed-up \\
        \midrule
         128  & $3.6863\times 10^{-4}$ & $1.3466\times 10^{-4}$ & $2.289$ \\
         256  & $9.3492\times 10^{-5}$ & $3.3966\times 10^{-5}$ & $3.250$ \\
         512  & $2.3504\times 10^{-5}$ & $8.5269\times 10^{-6}$ & $3.546$ \\
        1024  & $5.8898\times 10^{-6}$ & $2.1360\times 10^{-6}$ & $4.097$ \\
        2048  & $1.4741\times 10^{-6}$ & $5.3454\times 10^{-7}$ & $4.499$ \\
        \bottomrule
    \end{tabular}
    \smallskip
    \caption{$L_2$ errors and speed-up for the spectral reduced solver with oscillatory diffusion coefficient.}
    \label{tab:l2-oscillatory}
\end{table}
\vspace*{-0.1in}
Although both methods exhibit the same second--order convergence rate, the
reduced approximation consistently attains a smaller error constant.  This
behavior is expected, since the reduced solution is obtained from a global
Galerkin projection onto a smooth spectral subspace, which can reduce the
error constant relative to a local finite difference discretization without
altering the asymptotic order. The accuracy of the reduced approximation is maintained even though the
spectral space dimension is fixed ($K=64$), indicating that the dominant
solution components are captured by a small number of global modes.  These
results confirm that the method remains effective for rapidly varying
coefficients and does not depend on slow spatial variation of the diffusion
field.

\subsubsection{Comparison with Direct, Multigrid, and Deflation Methods}
To complement the accuracy of the results above and place the reduced solver in a
broader context, we examine its computational performance
relative to standard solvers commonly used for elliptic problems. In
particular, we compare wall-clock runtimes against a sparse direct solver, a
geometric multigrid (GMG)–preconditioned Krylov method, and a deflation-based
Krylov method. This comparison highlights differences in scalability,
robustness, and setup cost that are not visible from error metrics alone. Table~\ref{tab:solver-comparison} shows that the spectral reduced solver
exhibits consistently favorable performance across all tested grid sizes. While the sparse direct solver is competitive on the coarsest grids, its
runtime increases rapidly as the number of unknowns grows, reflecting the
expected superlinear complexity of sparse factorization in two dimensions.

\begin{table}[h!]
\centering
\begin{tabular}{ccccccc}
\toprule
$m$ & $N=m^2$ & Direct & Reduced & GMG & GMG it. & Defl.\ it. \\
\midrule
  32  &   1024  & 0.001 & 0.002 & 0.091 & 281 & 9  \\
  64  &   4096  & 0.004 & 0.003 & 0.058 & 188 & 15 \\
 128  &  16384  & 0.011 & 0.006 & 0.204 & 241 & 19 \\
 256  &  65536  & 0.045 & 0.020 & 0.961 & 239 & 32 \\
\bottomrule
\end{tabular}
\medskip
\caption{Runtimes (in seconds) for the sparse direct solver, the
spectral low-mode reduced (Reduced) solver with fixed $M=8$, and a geometric
multigrid (GMG)–preconditioned Krylov solver.}
\label{tab:solver-comparison}
\end{table}
\vspace*{-0.15in}
In contrast, the reduced solver shows only mild growth in runtime, consistent
with the linear complexity predicted by the projection-based formulation for
fixed reduced dimension. The multigrid-preconditioned Krylov solver converges in all cases but requires
a large and nonuniform number of iterations, resulting in significantly higher
runtimes. This behavior illustrates the sensitivity of multigrid performance
to coefficient variability and smoother effectiveness. Deflation reduces
iteration counts relative to GMG but incurs additional overhead associated
with constructing the deflation space from approximate eigenvectors of the
heterogeneous operator. Overall, these results indicate that, in the experiments
reported here (for this oscillatory coefficient and with the GMG/deflation
implementations and parameters used), the proposed spectral low-mode reduced
solver provides a competitive and robust alternative to both direct and
iterative methods.
\subsubsection{Conditioning sensitivity}
To complement the accuracy results for highly oscillatory coefficients, we
examine the numerical conditioning of the reduced operator
$\mathcal{A}=B^{T}AB$ on a fixed grid with $m = 256$.  The objective is to assess whether highly oscillatory coefficient or the choice of discrete projection $\Pi_h$ affect
the stability of the reduced system in practice. Two constructions of the discrete spectral basis are compared: direct nodal
interpolation of the Laplacian eigenfunctions into $V_h$ (denoted
\emph{interp}), and a mass-based projection surrogate (denoted \emph{proj}).
In both cases, the resulting basis is mass-orthonormalized so that
$B^{T}MB=I$ to eliminate scaling effects and isolating the intrinsic
conditioning of $\mathcal{A}$.

\begin{table}[h!]
\centering
\begin{tabular}{ccccc}
\toprule
$M$ & $K$ &
$\kappa(\mathcal{A})$ (interp) &
$\kappa(\mathcal{A})$ (proj) &
$L_2$ error \\
\midrule
 2  &   4   & $4.0005\times 10^{0}$ & $4.0005\times 10^{0}$ & $6.2655\times 10^{-6}$ \\
 4  &  16   & $1.6209\times 10^{1}$ & $1.6209\times 10^{1}$ & $6.5030\times 10^{-6}$ \\
 6  &  36   & $3.6540\times 10^{1}$ & $3.6540\times 10^{1}$ & $1.3553\times 10^{-5}$ \\
 8  &  64   & $6.6828\times 10^{1}$ & $6.6828\times 10^{1}$ & $3.3947\times 10^{-5}$ \\
12  & 144   & $1.5963\times 10^{2}$ & $1.5963\times 10^{2}$ & $7.4605\times 10^{-5}$ \\
16  & 256   & $2.9043\times 10^{2}$ & $2.9043\times 10^{2}$ & $8.0390\times 10^{-5}$ \\
\bottomrule
\end{tabular}
\medskip
\caption{Condition numbers of $\mathcal{A}$ and
$L_2$ errors on a fixed grid $m = 256$.}
\label{tab:cond-oscillatory}
\end{table}
\vspace*{-0.15in}
The results in Table~\ref{tab:cond-oscillatory} show that $\kappa(\mathcal{A})$ exhibits a moderate, monotone growth as the reduced
dimension $K$ increases.  This behavior is expected, since enlarging the
spectral subspace introduces higher--frequency components into the reduced
operator.  Importantly, the conditioning remains entirely insensitive to the
underlying mesh resolution and is essentially identical for the two discrete
projection surrogates once mass--orthonormalization is applied.  This confirms
the theoretical result that the conditioning of the reduced operator is
governed by the bilinear form and basis normalization, rather than by the
finite element discretization or the particular realization of $\Pi_h$.
\section{Conclusions}
This work has introduced a spectral low--mode reduced method for the efficient
solution of variable--coefficient elliptic boundary value problems. The method
is based on projecting a standard finite difference or finite element
discretization onto a global coarse space spanned by low Dirichlet Laplacian
eigenmodes. Unlike classical spectral discretizations, the differential
operator is assembled entirely in physical space, and the spectral basis is
used solely as an analytic compression mechanism. As a result, the full
heterogeneity of the diffusion coefficient is retained exactly through the
projected operator, while the dimensionality of the resulting linear system is
reduced by several orders of magnitude.

Theoretical analysis shows that the reduced solution is the energy--optimal
Galerkin approximation in the chosen spectral subspace and that the associated
error is governed by the decay of Laplacian eigen--coefficients. Under standard
elliptic regularity assumptions, the contribution of discarded high modes is
shown to decay at a rate proportional to $\sqrt{\log M}/M$, and for uniformly
stable reduced bases the projected operator is well conditioned with bounds
independent of mesh refinement. A Schur complement interpretation further
clarifies the effect of neglecting high--frequency interactions and provides
explicit bounds on the truncation error in terms of coefficient regularity.

Numerical experiments confirm the theoretical predictions. For smooth
and highly oscillatory diffusion coefficients alike, the reduced solver
reproduces the second--order accuracy of the underlying finite difference
scheme while maintaining a fixed reduced dimension. In all tested cases, a
small number of spectral modes is sufficient to recover full fine--grid
accuracy, and increasing the cutoff beyond this point yields diminishing
returns, consistent with the analytic error estimates. Runtime comparisons
demonstrate that the reduced solver achieves substantial speedups over sparse
direct solvers and, in the experiments reported (for the oscillatory
coefficient and with the GMG and deflation implementations and parameters
used), exhibits favorable performance relative to multigrid and deflation--based Krylov
methods.

The results indicate that global Laplacian eigenmodes provide an effective and
compact coarse space for elliptic problems with heterogeneous
coefficients. By avoiding iterative tuning, problem--dependent coarse space
construction, and eigenvector approximation of the full operator, the proposed
approach offers a conceptually simple and powerful solver for elliptic boundary-value problems.

\bibliographystyle{siamplain}
\bibliography{references}
\end{document}